\documentclass[12pt]{article}
\usepackage[]{times}
\usepackage{amsmath}
\usepackage{paralist}
\usepackage{fullpage}
\usepackage{amssymb}
\usepackage{mathrsfs}
\usepackage{bm}
\usepackage{graphicx}
\usepackage{graphicx}
\usepackage{wrapfig}
\usepackage{epsfig}
\usepackage{url}
\usepackage{fullpage}
\usepackage{psfrag}
\usepackage{verbatim}
\usepackage{amsthm}
\usepackage[]{times}
\usepackage{amsmath}
\usepackage{paralist}
\usepackage{fullpage}
\usepackage{amssymb}
\usepackage{amsrefs}

\usepackage[colorlinks=true, pdfstartview=FitV, linkcolor=blue,
            citecolor=blue, urlcolor=blue]{hyperref}
\usepackage[usenames]{color}
\definecolor{Red}{rgb}{0.7,0,0.1}
\definecolor{Green}{rgb}{0,0.6,0}

\usepackage{mathrsfs}
\usepackage{bm}
\usepackage{graphicx}
\usepackage{graphicx}
\usepackage{wrapfig}
\usepackage{epsfig}
\usepackage{url}
\usepackage{fullpage}
\usepackage{psfrag}
\usepackage{verbatim}

\numberwithin{equation}{section}
\newtheorem{thm}{Theorem}[section]
\newtheorem{lemma}{Lemma}[section]
\newtheorem{cor}{Corollary}[section]
\newtheorem{prop}{Proposition}[section]

\newtheorem{remark}{Remark}[section]

\title{A Delayed Yule Process}
\author{Radu Dascaliuc\thanks{Department of Mathematics,  Oregon State University, Corvallis, OR, 97331. {dascalir@math.oregonstate.edu}}
\and
Nicholas Michalowski
 \thanks{Department of Mathematics,  New
    Mexico State University, Las Cruces, NM,  88003.}
\and Enrique Thomann\thanks{Department of Mathematics,  Oregon State University, Corvallis, OR, 97331. }
\and
Edward C. Waymire\thanks{Department of Mathematics,  Oregon State University, Corvallis, OR, 97331. {waymire@math.oregonstate.edu}.}
}

\begin{document}

\maketitle

\begin{abstract}
In now classic work, David Kendall (1966) recognized that the Yule process and Poisson process could be related by a (random) time change.  Furthermore, he showed that the Yule population size
 rescaled by its mean has an almost sure
exponentially distributed limit as $t\to \infty$. 
In this note we introduce a class of coupled delayed Yule processes
parameterized by $0 < \alpha \le 1$ that includes the Poisson process
at $\alpha = {1/2}$.  Moreover we extend Kendall's limit theorem
to include a larger class of positive martingales derived from functionals that gauge the population genealogy.  A somewhat surprising connection 
with the Holley-Liggett smoothing transformation also emerges in this context.  Specifically, the latter is exploited to uniquely characterize the 
moment generating functions 
of distributions of the limit martingales, generalizing Kendall's mean one exponential limit.
\end{abstract}

\section{Introduction}  
The {\it basic Yule process} $Y = \{Y_t:t\ge 0\}$
 is a continuous time
 branching process starting from
a single progenitor in which a particle survives for a mean
one, exponentially distributed time before being replaced by two 
offspring independently evolving in the same manner.  $Y_t$
represents the size of the population of particles at time $t\ge 0$,
starting from $Y_0 = 1$.  
The {\it basic Poisson process} $N = \{N_t:t\ge 0\}$ is another
continuous time Markov process in which a particle survives for a mean one,
exponentially distributed time before being replaced by a single
particle that evolves in the same manner.  The shift 
$N_t+1$ represents the number of replacements that have occurred
by  time $t \ge 0$, $N_0 = 0$.  The multiplicative (geometric) growth
of the process $Y$ is in stark contrast to the additive growth of 
$N$.

Considerations of 
evolutionary processes, to be referred to as {\it delayed
Yule processes,} arise somewhat naturally in the probabilistic analysis of 
quasi-linear evolution equations such as incompressible Navier-Stokes
equations, and complex Burgers equation by probabilistic
 methods originating
with Le Jan and Sznitman \cite{YLJ_AS_1997}.  In particular, considerations of 
non-uniqueness and/or explosion problems in \cite{RD_NM_ET_EW2015}
for this framework prompted the present considerations.  However this
paper has a purely probabilistic focus and does not depend on such
motivations.   In fact, the probabilistic framework may also 
be of interest in the context of evolutionary biological processes. 

The 
principal
results are extensions of the aforementioned theorems
of Kendall (see \cite{DK1966}).  The 
connection between the Poisson and Yule process is
given by an exact coupling of the two processes through a binary
tree-indexed family of i.i.d. exponential random variables defined on
a probability space $(\Omega,{\cal F},P)$.  Precise criteria for the 
uniform integrability of positive martingales derived from
a family of gauges of the genealogy of the 
Yule process, including cardinality, is also given. The exact
 limit distribution
is identified for these uniformly integrable 
martingales as unique (mean one) fixed points of the Holley-Liggett
smoothing
operator \cite{RH_TL1981}, 
generalizing Kendall's mean one exponential limit 
in the case the
gauge is cardinality of the population. The latter is the 
Gamma distributed fixed point corresponding to the uniform (Beta)
smoothing factor.
 Finally, it is
shown that $\alpha = {1/2}$ is a critical transition value between 
bounded and unbounded infinitesimal generators defining the
$\alpha$-delayed Yule processes for $0<\alpha \le 1$.

\section{Delayed Yule Process}
To begin, consider  the modification of the Yule process given by
successively halving the previous branching
frequencies, i.e., doubling the mean holding time of particles  
of each generation. 
That is, let  $\{T_v: v\in{\bf T}=\cup_{k=0}^\infty\{1,2\}^k\}$,
with $\{1,2\}^0 = \{\theta\}$,  be a binary,
tree-indexed family  of i.i.d. mean one exponentially distributed random
variables rooted at a single progenitor $\theta$, and define
$$V^{(\frac{1}{2})}(t) = \Big\{v\in{\bf T}:\ \sum_{j=0}^{|v|-1}({1}/{2})^{-j}T_{v|j} \le t < \sum_{j=0}^{|v|}({1}/{2})^{-j}T_{v|j}\Big\},
\quad t\ge 0,$$
where $|\theta| = 0$, and $|v| = |<v_1,\dots,v_k>| = k$ denotes the
{\it height} of vertex $v\in{\bf T}$. Also $v|j = <v_1,\dots,v_j>$ is the restriction
of $v$ to generation $j\le k$.  Also, by convention, $\sum_{j=0}^{-1} = 0$.

Observe that
$$Y_t = \#V^{(1)}(t) = \Big\{v\in{\bf T}:\ \sum_{j=0}^{|v|-1}T_{v|j} \le t < \sum_{j=0}^{|v|}T_{v|j}\Big\},
\quad t\ge 0,$$
defines the basic Yule process; throughout $\#V$ will denote the cardinality of a set $V$.

Let $\tau_k, k = 1,2,\dots$ be the increasing
sequence of jump times of the {\it $\frac{1}{2}$-delayed Yule process}
defined by
$$N_t = \#V^{(\frac{1}{2})}(t)-1, t\ge 0.$$

\begin{lemma}[Key Coupling Lemma 1]\label{KCL1}
 For arbitrary $k\ge 1$,
 conditionally given $\tau_0 = 0,
 \tau_1,\dots, \tau_{k-1}$, $\tau_k-\tau_{k-1}$ is exponentially
 distributed with mean one. In particular,
 $\tau_k- \tau_{k-1}, k = 1,2,\dots$ is an i.i.d. sequence.
 \end{lemma} 
\begin{proof}
 First observe that $\tau_1=T_\theta$, and thus  $P(\tau_1 > t) = e^{- t}, t\ge 0$.
 Next $P(\tau_2-\tau_1 > t) = P(2T^{(1)}\wedge 2T^{(2)} > t)
 = e^{- {\frac{t}{2}}}e^{- {\frac{t}{2}}} = e^{- t}$.  More
 generally for $k\ge 2$, an induction argument shows that
 given $\tau_1,\dots,\tau_{k-1}$, 
 $\tau_k- \tau_{k-1}$ is the minimum of $k$ independent exponentially
 distributed random variables whose intensities add
 to one. To see this, for $k\ge 2$, on $[\tau_1 \le t]$,
 express the process $V^{(\frac{1}{2})}(t), t\ge  \tau_1,$
 as the disjoint union of two independent, sets of 
 vertices $V_{(j)}^{(\frac{1}{2})}(t-T_\theta), t > 0$, $j = 1,2$,
Then  $\tau_k-\tau_{k-1}$ is the minimum of the left and
right independent jump times.  In view of the scaling of
the holding times by a factor of $2$ in successive generations
in the definition of $V^{(\frac{1}{2})}$, it follows
by induction that this left-right minimum 
is the minimum of two independent
exponential holding times with intensity $\frac{1}{2}$,
respectively, and therefore exponential with unit intensity.
 \end{proof}

\begin{thm}\label{poissonhalf}
 The stochastic process 
$N_t = \#V^{(\frac{1}{2})}(t)-1, 
t\ge 0$, is a Poisson process with unit intensity.
\end{thm}
\begin{proof}
This is a direct consequence of the key coupling lemma1, making
$N$ a process with stationary independent increments, 
$N_0 = 0$, and
 $P(N_t = k) = P(\tau_k\le t < \tau_{k+1})
= \frac{t^k}{k!}e^{- t}, t\ge 0, k = 0,1,2,\dots$.
\end{proof}

Replacing $\frac{1}{2}$ by a parameter $\alpha\in (0,1]$
 in successive
generations of the basic Yule process defines the 
{\it $\alpha$-delayed Yule process}.  Namely,
$$V^{(\alpha)}(t) = \Big\{v\in{\bf T}: \sum_{j=0}^{|v|-1}\alpha^{-j}T_{v|j} \le t < \sum_{j=0}^{|v|}\alpha^{-j}T_{v|j}\Big\},
\quad t\ge 0.$$

Accordingly, $V^{(\alpha)}$ is a continuous time jump Markov process taking value in the (countable) space ${\cal E}$ of {\it evolutionary sets} defined
inductively by $V\in{\cal E}$ if and only if $V$ is a {\it finite} subset
of ${\bf T} = \cup_{n=0}^\infty\{1,2\}^n$,  such
that 
 $$ V = 
\left\{
\begin{array}{ll}
\{\theta\} & {\rm{if}}~~ \#V = 1, \\
W\backslash\{w\}\cup\{<w1>,<w2>\} & {\rm{for\  some}}~~~W\in{\cal E},\ \#W = \#V -1,\  w\in W,\ 
\mbox{else.}
\end{array}
\right.
$$

Although one may check that $V^{(\alpha)}$ is a 
Markov process on ${\cal E}$, the functional
 $\#V^{(\alpha)}$ is {\it not} generally Markov; exceptions
 being $\alpha = \frac{1}{2}, 1.$ 
  When $\alpha=1$, $\#V^{(\alpha)}$ is the classical 
 Yule process, and so  it is obviously Markov, while the 
 case $\alpha =\frac{1}{2}$
 is made special in a way already exploited in the proof of the Key Coupling Lemma \ref{KCL1}. 
The Markov property is a consequence of the following lemma that can be obtained by a simple induction argument 
 left to the reader.
 
 \begin{lemma}[Key Coupling Lemma 2]\label{KCL2}
 \label{onehalf}
  For any $V\in{\cal E}$ one has
 $$\sum_{v\in V}\left({1}/{2}\right)^{|v|}  = 1.$$
 \end{lemma}
 
  In addition to cardinality, letting $\beta > 0$,
 the following functionals serve to gauge the genealogy
 of the evolution:
 \begin{equation}
 \label{gaugedef}
 a_\beta(V) = \sum_{v\in V}\beta^{|v|}, \quad V\in{\cal E}.
 \end{equation}
By the Key Coupling Lemma \ref{KCL2}, one has that
$a_{1/2}(V) = 1$ for all $V\in{\cal E}$.
The cardinality $\#V$
 is covered by $\beta = 1$, and the following provides a generalization
 of Kendall's classic limit theorem to other {\it gauges} of the 
 genealogical structure of the Yule process. 

 \begin{thm}\label{unif_int_thm}
For each $\beta\in (0,1]$,  $A_\beta(t) = e^{-(2\beta -1)t}a_\beta(V^{(1)}(t)),
t\ge 0$, is a positive martingale.  Moreover, $A_\beta$ is uniformly integrable
if and only if $\beta\in (\beta_c,1]$ where 
$\beta_c\approx 0.1866823$
 is the unique in $(0,1]$ solution to
 \begin{equation}\label{beta_c}
\beta_c\ln\beta_c = \beta_c-{\frac{1}{2}}.
\end{equation}
\end{thm}
\begin{proof}
Let $m_\beta(t) = {\mathbb E}a_\beta(V^{(1)}(t)), t\ge 0$. 
First, let us check that 
\begin{equation}\label{Ea}
m_\beta(t) = e^{(2\beta -1)t},\quad
t\ge 0.
\end{equation} 
For this write
\begin{equation}
\label{basicrecursion}
a_\beta(V^{(1)}(t)) = 1[T_\theta > t] + 1[T_\theta\le t]\beta\{
a_\beta(V^{(1)+}(t-T_\theta)) + a_\beta(V^{(1)-}(t-T_\theta))\},
\end{equation}
where $V^{(1)\pm}(t-T_\theta)$ are conditionally independent 
copies of $V^{(1)}$ given $T_\theta$.  Taking expected values one has
$$m_\beta(t) = e^{-t} + 2\beta\int_0^t e^{-s}m_\beta(t-s)ds, 
\quad m_\beta(0) = 1.$$
The expression (\ref{Ea}) now follows.  

To establish the martingale property, let $0\le s < t$ and write
$$a_\beta(V^{(1)}(t)) = \sum_{w\in V^{(1)}(s)}\sum_{v\in V^{(1),w}(t-s)}
\beta^{|w|}\beta^{|v|},$$
where $V^{(1),w}$ are the delayed Yule processes rooted at  
$w\in V^{(1)}(s)$. 
Note that 
the respective processes $V^{(1),w}, \ w\in V^{(1)}(s),$ 
 are 
conditionally independent
given $V^{(1)}(s)$, and therefore
$${\mathbb E}[e^{-(2\beta -1)t}a_\beta(V^{(1)}(t)) | {\cal F}_s] = 
e^{-(2\beta -1)t}m_\beta(t-s)a_\beta(V^{(1)}(s))
= e^{-(2\beta-1)s}a_\beta(V^{(1)}(s)).$$
Thus $A_\beta$ is a positive martingale.  So,
by the martingale convergence theorem, it follows that
$$A_\beta(\infty) = \lim_{t\to\infty}e^{-(2\beta-a)t}a_\beta(V^{(1)}(t)),$$
exists almost surely.   Moreover, from (\ref{basicrecursion}) one has the distributional 
recursion
\begin{equation}
\label{distrecursion}
A_\beta(\infty) = \beta e^{-(2\beta-1)T_\theta}(A_\beta^+(\infty)
+ A_\beta^-(\infty)).
\end{equation}

Let us first investigate parameters 
$\beta\in (0,1]$ such that $A_\beta(\infty) = 0$ almost surely.
For this let $h\in (0,1)$ and observe that, 
since $(x+y)^h\le x^h+y^h$ and ${\mathbb E}(e^{-\delta T_{\theta}})
=1/(1+\delta)$, (\ref{distrecursion}) yields
$${\mathbb E}A_\beta^h(\infty) \le 2\beta^h\frac{1}{1+(2\beta-1)h}
{\mathbb E}A_\beta^h(\infty), \quad 0 < h < 1.$$
Thus, if $A_\beta(\infty) > 0$ with positive probability, then
\begin{equation}\label{beta_ineq}
\frac{2\beta^h}{1+(2\beta-1)h}\ge 1, \quad 0 < h < 1.
\end{equation}
By comparing the
 functions $\phi(h)=\beta^h$ and $\psi(h)=1+(2\beta-1)h$ on $h\in[0,1]$, 
it follows that (\ref{beta_ineq}) holds if and only if 
$$\beta\ge\beta_c,$$
where $\beta_c\approx 0.1866823$ is the unique solution on $(0,1]$ to
 the equation 
$
2\beta_c\ln\beta_c = (2\beta_c-1).
$
Then $\beta < \beta_c$ implies $A_\beta(\infty) = 0$ almost
surely.

For the converse, i.e., uniform integrability of 
the positive martingale $\{A_\beta(t): t\ge 0\}$, we will use
an inequality from \cite{JN1988}, attributed there
 to B. Chauvin and J. Neveu, especially suited for such
 problems.  For present purposes, if $1< p\le 2$, and
  $X_1,X_2\in L^p(\Omega, {\cal F}, P)$ are independent,
   positive random variables, then 
 \begin{equation}\label{C-N}
v_p(X_1+X_2) \le v_p(X_1) + v_p(X_2),
\end{equation}
 where 
   $v_p(X_j) = {\mathbb E}X_j^p - ({\mathbb E}X_j)^p, j=1,2$.

By the basic recursion (\ref{basicrecursion}), one has
\begin{equation}\label{EAp}
{\mathbb E}A_\beta^p(t) = e^{-[(2\beta-1)p+1]t}
+\beta^p\int_0^te^{-[(2\beta-1)p+1]s}{\mathbb E}(A_\beta^+(t-s)
+A_\beta^-(t-s))^p ds.
\end{equation}
Applying (\ref{C-N}) and using the submartingale property ${\mathbb E}A^p_\beta(t-s)
\le {\mathbb E}A^p_\beta(t), 0\le s\le t$ together with the fact that ${\mathbb E}A_\beta(t-s) = 1$,
we estimate
$$\begin{aligned}
{\mathbb E}(A_\beta^+(t-s)
+A_\beta^-(t-s))^p &=
v_p(A_\beta^+(t-s)+ A_\beta^-(t-s))+ ({\mathbb E}(A_\beta^+(t-s))+ A_\beta^-(t-s))^p\\
&\le v_p(A_\beta^+(t-s))+v_p(A_\beta^-(t-s))+2^p({\mathbb E}(A_\beta(t-s)))^p\\
&\le 2{\mathbb E}A^p_\beta(t-s)+2^p\le 2{\mathbb E}A^p_\beta(t)+2^p,
\end{aligned}$$
Thus, (\ref{EAp}) yields
$$
{\mathbb E}A_\beta^p(t)\le e^{-[(2\beta-1)p+1]t} + \frac{(2{\mathbb E}A^p_\beta(t)+2^p)\beta^p}{
(2\beta-1)p+1},
$$
which implies
$$\frac{(2\beta-1)p + 1 -2\beta^p}{(2\beta-1)p+1}{\mathbb E}A_\beta^p(t)
\le e^{-[(2\beta-1)p+1]t} + \frac{(2\beta)^p}{(2\beta-1)p+1},
\quad t\ge 0.$$
In particular, uniform integrability follows under the condition 
that for some $p\in (1,2]$,
$$(2\beta-1)p + 1 -2\beta^p > 0.$$
Equivalently, $\beta > \beta_c$ where, 
as before, $\beta_c$ -- the solution of (\ref{beta_c}).

 To complete the proof requires
consideration of the case $\beta = \beta_c$.  If, for sake of contradiction, one assumes 
uniform integrability then,  
as is elaborated in the proof of the Proposition \ref{H-L_prop} below,
 the distribution of $A_{\beta_c}(\infty)$
provides a mean one fixed point to the Holley-Liggett smoothing
map, see \cite{RH_TL1981}, where it is shown that there
is not  a mean one fixed point at $\beta_c$.
\end{proof}

For $\beta \in [0,1]$, define the moment generating function 
$$\varphi_\beta(r) = {\mathbb E}e^{-rA_\beta(\infty)}, \quad r\ge 0,$$
where $A_\beta(\infty) = \lim_{t\to\infty}A_\beta(t).$
Note that by Proposition \ref{unif_int_thm}, 
\[\varphi_\beta'(0)=0\quad \mbox{if}\ \beta<\beta_c \quad\mbox{and}\quad \varphi_\beta'(0)=-1\quad \mbox{if}\ \beta>\beta_c\]

Also define a probability measure $\nu_\beta$ on $S_\beta$
 where 
 $S_\beta = [0,\beta]$ for $\beta > 1/2$, and 
 $S_\beta = [\beta,\infty)$ for $0< \beta < 1/2$, and
 \begin{equation}\label{nu's}
 \nu_{\frac{1}{2}}(ds) = \delta_{\frac{1}{2}}(ds),\quad
 \nu_\beta(ds) = \frac{\left({s}/{\beta}\right)^\frac{1}{2\beta -1}}{|2\beta-1|}\frac{ds}{s},
 \ \beta\neq \frac{1}{2}.\end{equation}

\begin{prop}\label{H-L_prop} For $\beta > \beta_c$, $\varphi_\beta$ is uniquely 
determined within the class of probability distributions on 
$[0,\infty)$
whose moment generating function satisfies
\begin{equation}\label{phi-eq}
\varphi_\beta(r) = \int_{S_\beta}\varphi_\beta^2(rs)\nu_{\beta}(ds), 
 \quad r \ge 0,\end{equation}
 such that 
 $\varphi_\beta(0) = 1$, $\varphi_\beta^\prime(0) = -{\mathbb E}A_\beta(\infty)$.
Equivalently, $\varphi_\beta$ is uniquely determined by the
delayed differential equation
\begin{equation}\label{phi_ODE}
\varphi_\beta'(r)=\frac{1}{r}\frac{1}{2\beta-1}\,\varphi_\beta^2(\beta r)-\frac{1}{r}\frac{1}{2\beta-1}\,\varphi_\beta(r),\quad \beta\in[0,1]\setminus\Big\{\frac{1}{2}\Big\}, 
\end{equation}
and the given initial conditions.
\end{prop}

\begin{proof}
First we will show that (\ref{phi-eq}) holds for $\beta\in[0,1]$. When $\beta=1/2$, by (\ref{distrecursion}),
\begin{equation}\label{phi_1/2}\varphi_{\frac{1}{2}}(r)=\varphi_\frac{1}{2}^2(r/2),\end{equation}
and thus (\ref{phi-eq}) holds with $\nu_{1/2}$ -- the Dirac measure as in (\ref{nu's}).
For $\beta\ne 1/2$, using the stochastic recursion (\ref{distrecursion}), we obtain:
\[
\begin{aligned}
\varphi_\beta(r) &= {\mathbb E }\big(e^{-rA_\beta(\infty)}\big)={\mathbb E}\left(
\exp\left[-r\beta e^{-(2\beta-1)T_\theta}\left(A_\beta^+(\infty))+A_\beta^-(\infty)\right)
\right]\right)\\
&= \int\limits_0^\infty e^{-t}\,{\mathbb E}\exp\left[-r\beta e^{-(2\beta-1)t}\left(A_\beta^+(\infty))+A_\beta^-(\infty)\right)
\right]\,dt\\
&= \int\limits_0^\infty e^{-t}\varphi_\beta^2\left(r\beta e^{-(2\beta-1)t}\right)\,dt.
\end{aligned}
\]
Now (\ref{phi-eq}) follows by the change of variables $s=\beta e^{-(2\beta-1)t}$.

For $\beta>\beta_c$, in view of the uniform integrality (see Theorem \ref{unif_int_thm}) one has ${\mathbb E}A_\beta(\infty) = 1$,  and we may use early results of \cite{RH_TL1981} on smoothing
transformations. Specifically, it is simple to check
 that for $\beta_c<\beta\le 1$, the random variable
$W_\beta = 2\beta e^{-(2\beta-1)T_\theta}$ has mean one (in fact, $\frac{1}{2}W_\beta$ is a re-scaling of the distribution $\nu_\beta$), while the recursion (\ref{distrecursion}) takes form
$$A_\beta(\infty) = W_\beta\Big(\frac{1}{2}A_\beta^+(\infty) + 
\frac{1}{2}A_\beta^-(\infty)\Big),$$
of a Holley-Liggett smoothing transformation within the framework of Theorem 7.1 in \cite{RH_TL1981}. 
Accordingly, the distribution of $A_\beta(\infty)$ is 
the unique positive mean one solution to the stochastic recursion provided
\[
{\mathbb E}(W_\beta \ln W_\beta)<\ln 2. 
\]
A direct calculation shows that ${\mathbb E}(W_\beta \ln W_\beta)=\ln(2\beta)-\frac{2\beta-1}{2\beta}$, and thus the inequality above is satisfied if and only if $\beta>\beta_c$.

To establish (\ref{phi_ODE}) we may use (\ref{phi-eq}), as follows (noting
that the implied 
differentiability is a property of a moment generating function of
a probability distribution on $[0,\infty)$):
\[
\varphi_\beta'(r)=\int\limits_{S_\beta}\frac{d}{dr}\varphi_\beta^2(rs)\nu_\beta(ds)=\frac{1}{r}\int\limits_{S_\beta}\frac{d}{ds}\varphi_\beta^2(rs)\,s\,\nu_\beta(ds).
\]
Now use (\ref{nu's}) and integrate by parts. In the case $\beta<1/2$ we get:
\[
\begin{aligned}
\varphi_{\beta}'(r) &=\frac{1}{r}\int\limits_\beta^{\infty}\frac{d}{ds}\varphi_{\beta}^2(rs)\, \frac{\left({s}/{\beta}\right)^\frac{1}{2\beta -1}}{1-2\beta}\,{ds}
=\frac{1}{r}\left. \varphi_{\beta}^2(rs)\, \frac{\left({s}/{\beta}\right)^\frac{1}{2\beta -1}}{1-2\beta}\right|_{s=\beta}^{\infty}+\frac{1}{r}\int\limits_\beta^{\infty}\varphi_{\beta}^2(rs)\, \frac{\left({s}/{\beta}\right)^\frac{1}{2\beta -1}}{(1-2\beta)^2}\,\frac{ds}{s}\\
&=-\frac{1}{r}\frac{1}{1-2\beta}\,\varphi_{\beta}^2(\beta r)+\frac{1}{r}\frac{1}{1-2\beta}\,\varphi_{\beta}(r),
\end{aligned}
\]
which implies (\ref{phi_ODE}) for $\beta\in[0,1/2)$. The case $\beta\in(1/2,1]$ is treated analogously.
\end{proof}

\begin{remark}
{\rm While the martingale limit is clearly a fixed point of the
Holley-Liggett smoothing transformation for any $\beta\in (0,1]$, 
the proof of 
uniform integrability appears to be essential to the identification of 
the critical parameter $\beta_c$ for a positive martingale
limit since fixed point uniqueness theorem is within the class
of mean one probability distributions on $[0,\infty)$. As noted
in \cite{RH_TL1981} for particular Beta distributions of $W$,
the fixed point distribution is a Gamma distribution.  This includes
the case of Kendall's theorem, \cite{DK1966}, for $\beta = 1$ in which $W$
is uniform on $(0,1)$ and the martingale limit has a mean-one
exponential distribution as given below.}
\end{remark}


\begin{cor}[Kendall's theorem]\label{Kendall_thm} $A_1(t) = e^{-t}\,Y_t, t\ge 0,$ is a
uniformly integrable martingale, and 
$A_1(\infty) = \lim_{t\to\infty}A_1(t)$ is exponentially distributed
with mean one. 
\end{cor}
\begin{proof}
It is easy to see that the mean one exponential moment generating function $1/(1+r)$ satisfies (\ref{phi-eq}) in case $\beta=1$. Now the 
fact that the exponential is indeed the distribution of $A_1(\infty)$ follows from the uniqueness statement of Proposition \ref{H-L_prop}.
\end{proof}


\begin{remark}{\rm
One can also 
obtain Kendall's result directly from (\ref{phi_ODE}). Indeed, when 
$\beta=1$ we have
\[(r\varphi_1(r))'=\varphi_1^2(r),\qquad\varphi_1(0)=1,\ \varphi_1'(0)=-1,\]
The non-zero solutions of the equation above can be obtained explicitly as
\[\varphi_1(r)=\frac{1}{1+c_0r},\]
while by the initial data, $c_0=1$, proving that the mean one exponential 
moment generating function is the only solution, and thus implying the aforementioned 
Kendall's theorem stated in  Corollary \ref{Kendall_thm}.
}
\end{remark}

\section{Infinitesimal Generator and another Critical Value for the Delayed Yule Process}
Give ${\cal E}$ the discrete topology and let $C_0({\cal E})$ 
denote the space of (continuous) real-valued functions 
$f:{\cal E}\to {\mathbb R}$ that vanish at infinity; i.e., given $\epsilon > 0$, 
one has $|f(V)| < \epsilon$ for all but finitely many $V\in{\cal E}$.
The subspace 
$C_{00}({\cal E})\subset C_0({\cal E})\subset L^{^{\infty}}({\cal E})$ of functions
with compact (finite) support is clearly dense in $C_0({\cal E})$
 for the uniform norm.

The construction at the 
outset of the coupled 
stochastic
processes
$V^{(\alpha)}, 0<\alpha \le 1$, provides corresponding semigroups of
positive linear contractions  $\{T_t^{(\alpha)}: t\ge 0\}$ 
defined by
$$T_tf(V) = {\mathbb E}_Vf(V^{(\alpha)}(t)), 
\quad t\ge 0, f\in C_0({\cal E}),$$
with the usual branching process convention that given
$V^{(\alpha)}(0) = V\in{\cal E}$, $V^{(\alpha)}(t)$ is the 
total progeny independently produced by single
progenitors at each $v\in V$. In fact, one may consider the
semigroup as defined
 on $L^{^{\infty}}({\cal E})\supset C_0({\cal E})$.

The usual considerations imply that the infinitesimal generator
$(L^{(\alpha)},{\cal D}_\alpha)$ of $V^{(\alpha)}$
is given on  
$C_{00}({\cal E})$ via
$$L^{(\alpha)} f(V) = \sum_{v\in V}\alpha^{|v|}\{f(V^v)
- f(V)\}, \quad f\in C_{00}({\cal E}),$$
where
$$V^v = V\backslash\{v\}\cup\{<v1,v2>\}, \quad v\in V.$$
One may naturally pursue the computation of a core for
$L^{(\alpha)}$, however for the present purposes the above is sufficient
to establish the following distinct role of $\alpha = \frac{1}{2}$
as a critical parameter.

\begin{prop} $(L^{(\alpha)},{\cal D}_\alpha)$, ${\cal D}_\alpha\subset L^{^{\infty}}({\cal E})$ -- the domain of $L^{(\alpha)}$, is a bounded linear operator
if and only if $\alpha\le \frac{1}{2}$.
\end{prop}
\begin{proof}
The sufficiency follows from the key coupling lemma 2, 
since for $\alpha\le \frac{1}{2}$ one has the bound $\sum_{v\in V}\alpha^{|v|} \le \sum_{v\in V}2^{-|v|} = 1,
V\in {\cal E}.$  In particular, for $f\in C_0({\cal E})$,
$$|L^{(\alpha)} f(V)| \le 2\sup_{W\in{\cal E}}|f(W)|, \quad V\in{\cal E}.$$
On the other hand, for $\alpha > \frac{1}{2}$, define a sequence
of functions $f_n\in C_00({\cal E})$ by
$$f_n(V) = h(V)\,{\bf 1}_{[h(V)\le n]}, \quad n = 1,2,\dots,$$
where $h(V) = \max\{|v|: v\in V\}, V\in{\cal E}$.  Then for 
full binary branching $h(V) = n, |V| = 2^n$. Thus
$\|f_n\|_{\infty} = n$, and for such $V$, 
 $$|L^{(\alpha)}f_n(V)| = \sum_{v\in V}\alpha^n = (2\alpha)^n.$$
 In particular 
 $$\frac{|L^{(\alpha)}f_n(V)|}{\|f_n\|_{\infty}} = \frac{(2\alpha)^n}{n}
 \to \infty\quad \mbox{as}\ n\to\infty\quad \mbox{for}\ \alpha >\frac{1}{2}.$$
\end{proof}

\begin{remark}
Although $a_\beta\notin C_0({\cal E})$ for any $\beta\in (0,1]$,
the following formal calculation
for $\alpha\in(0,1]$,
$$L^{(\alpha)} a_\beta(V) = (2\beta -1)a_{\alpha\beta}(V), \quad V\in{\cal E},
$$
is intriguing from the perspective of precise identification of the generator.
In particular, $a_\beta$ is formally a positive eigenfunction of $L^{(1)}$
with non-positive 
eigenvalue $2\beta-1 < 0$ for $\beta < \frac{1}{2}$ as required
for a contraction semigroup of positive linear operators.  To make this formal
calculation rigorous obviously requires a modification of the function space beyond
the standard choice $C_0({\cal E})$.
\end{remark}

Finally let us conclude by noting a closely related evolution
that takes place in sequence space that may be of interest
in other contexts.
  For $V\in{\cal E}$,  let 
$$g_k(V) = \#\{v\in V: |v| = k\}, \quad k = 0,1,2,\dots .$$
Also define an equivalence relation on ${\cal E}$ by 
$V\sim W$, \  $V,W\in{\cal E}$, if and only if $g_k(V) = g_k(W)$ for all $k$.
Then the space of equivalence classes ${\cal E}/\sim$ is in one-to-one correspondence with a subset of the sequence space
 $c_{00}({\mathbb Z}_{+})\subset \ell_1({\mathbb Z}_+)$ defined inductively as follows: 
 $n = (n_0,n_1,\dots)\in c_{00}({\mathbb Z}_{+})$ belongs to the
 space ${\cal E}_0$ of  {\it evolutionary sequences}
 if either $n = (1,0,\dots)$ or, otherwise, there is an 
 $m\in{\cal E}_0\subset
 c_{00}({\mathbb Z}_+)$
 such that $m = n^{(k)} := (n_0,n_1,\dots,n_{k}-1,n_{k+1}+2,n_{k+2},\dots)$ for
 some $k\ge 0$ such that $n_k\ge 1$.  Note that $\sum_{j=0}^\infty n_j
 = \sum_{j=0}^\infty m_j -1$.  
 For $0<\alpha \le 1$,  the equivalence relation induces
  $N^{(\alpha)} =\{N^{(\alpha)}(t): t\ge 0\}$
 as the continuous time jump Markov process on ${\cal E}_0$
 with
 generator given for $f\in C_{00}({\cal E}_0)$ by
 $$\tilde{L}^{(\alpha)}f(n) = \sum_{k=0}^\infty n_k\alpha^k(f(n^{(k)}) - f(n)),
 \quad n\in{\cal E}_0.$$
 
 \section{Acknowledgments}    This work was partially supported by
grants DMS-1408947, DMS-1408939, DMS-1211413, and DMS-1516487
from the National Science Foundation.

 \bibliographystyle{plain}
\bibliography{DY_bib}

 \end{document}